\documentclass[a4paper,twoside]{amsart}
\usepackage[T1]{fontenc}
\usepackage[utf8]{inputenc}
\usepackage[slantedGreeks, partialup, noDcommand]{kpfonts}

\usepackage{graphicx}

\newtheorem{lemma}{Lemma}[section]
\newtheorem{theorem}{Theorem}[section]
\newtheorem{corollary}{Corollary}[section]

\newcommand{\RR}{\mathbb{R}} 

\newcommand{\NN}{\mathbb{N}}

\newcommand{\D}{\mathop{}\!\mathrm{d}}	
\newcommand{\mup}{\mathrm{m}}
 
\newcommand{\Sup}{\mathbb{S}}
\newcommand{\Bup}{\mathrm{B}}

\title{Hyper contractivity on the unit circle for ultraspherical measures: linear case}  

\author{Paata Ivanisvili}
\address{Paata Ivanisvili,
Department of Mathematics,
University of California, Irvine\\
Irvine, CA 92697}
\email{pivanisv@uci.edu} 

\author{Alexander Lindenberger}
\address{A. Lindenberger, Institute of Analysis, Johannes Kepler University Linz\\Altenberger Strasse 69, 4040 Linz, Austria}
\email{alexander.lindenberger@jku.at}

\author{Paul F. X. M\"uller}
\address{P. F. X. M\"uller, Institute of Analysis, Johannes Kepler University Linz\\Altenberger Strasse 69, 4040 Linz, Austria}
\email{paul.mueller@jku.at}

\author{Michael~Schmuckenschl\"ager}
\address{M. Schmuckenschl\"ager, Institute of Analysis, Johannes Kepler University Linz\\Altenberger Strasse 69, 4040 Linz, Austria}
\email{michael.schmuckenschlaeger@jku.at}

 
\begin{document}

\maketitle

\begin{abstract}
  In this paper we extend complex uniform convexity estimates for $\mathbb{C}$  to $ \mathbb{R}^n $ and determine best constants.  Furthermore we
  provide the  link to log-Sobolev inequalities 
  and  hypercontractivity estimates for  ultraspherical measures.  
\end{abstract}

\maketitle

\section{Introduction}
The starting point of this paper is 
the Bonami's  sharp complex convexity estimate  (A. Bonami~\cite[Chapter III, Theorem 7]{bonami70})
\begin{equation} \label{inequality_bonami}
 \int_{\Sup^{1}} |x+ a\zeta |\D\mup(\zeta) \ge \left( |x|^2 + \frac 1 2 a^2 \right)^{\frac 1 2} \qquad \text{for }  x \in \RR^2, a\in [0,\infty), 
\end{equation}
where $\Sup^{1} $ denotes the unit circle in $\RR ^2 $  and $\mup$ the usual Haar measure on $\Sup^{1}$ with $\mup(\Sup^{1})=1$. 
W. J. Davis, D. J. H. Garling and N. Tomczak-Jaegermann  \cite[Proposition 3.1]{davis1984} presented a proof of \eqref{inequality_bonami} based on the power series representation of elliptic integrals. 
Separately S. Janson, applying F. Weissler's   logarithmic Sobolev inequalities on the circle, 
 obtained (in particular) that 
for any  $\alpha\in(0,2]$, 
 \begin{equation} \label{inequality_w}
 \left( \int_{\Sup^{1}} |x+ a \zeta |^\alpha \D \mup(\zeta) \right)^{\frac 1 \alpha} \ge (|x|+ \frac \alpha 2 a^2)^{\frac 1 2} \qquad \text{for }  x \in \RR^2, a\in [0,\infty), 
 \end{equation}
 and $\frac \alpha 2$ is the best (i. e. largest) real constant satisfying \eqref{inequality_w}.
 In 2007, Aleksandrov~\cite{aleksandrov07} presented an elegant short proof of \eqref{inequality_w}. The proofs in~\cite{weissler80}, respectively~\cite{aleksandrov07} of  \eqref{inequality_w} are complex analytic {in nature}; they   don't seem to work in  higher dimensions, where   $\Sup^{1} $ is replaced by the unit sphere in $\RR^n$
and  where $\mup$ is replaced by $\sigma$, the normalized Haar measure  on the unit sphere in $\RR^n$. 

Recently,
\cite{alex2019sharp} we recorded  a  proof of \eqref{inequality_w},  
based on Green's identities and  sub-harmonicity estimates such as,   
 $$\int_{\Sup^{1}} |x+ a \zeta |^{\beta} \D \mup(\zeta) \ge \max\{a ,|x|\}^{\beta}, \qquad \beta \in \RR (!), \quad  x \in \RR^2, a\in [0,\infty) . $$

In the present  paper we prove that the  extension of  \eqref{inequality_w} to dimensions $ n \ge 3 $ holds true. 
The cases  $ n = 3 $ and $ n \ge 4 $ are treated separately. For $ n = 3 $ we were able to adjust the argument in  \cite{alex2019sharp}. 
In dimensions   four and higher,  our proof is guided by  the link  between the integral-term  appearing in \eqref{inequality_thm}
and  
Riesz potential operators on $\RR ^n$, acting on the surface measure $\sigma$. 

In Section~\ref{ultra} we investigate the hypercontractivity for ultraspherical measures on the unit circle 
$$
d\nu_{m}(z) =c_{m} |\sin(\theta)|^{m}d\theta, \quad z=e^{i \theta} \in \mathbb{S}^{1}, \quad \nu_{m}(\mathbb{S}^{1})=1, \quad m > -1 ,
$$
considering the ``linear polynomials''  on $\mathbb{S}^1$  given by  $ f(z) = a+bz $.

For $m=-1$ by definition we set $d\nu_{-1}(z) = \frac{1}{2}(\delta_{1}(z)+\delta_{-1}(z))$. 
We are interested in real numbers $m,p,q,r$ with $0<p\leq q<\infty$ and $r \in \mathbb{R}$ such that
\begin{align}\label{uhyp}
\| 1+rbz\|_{L^{q}(\mathbb{S}^{1}, d\nu_{m})} \leq \| 1+bz\|_{L^{p}(\mathbb{S}^{1}, d\nu_{m})} \quad \text{for all} \quad b \in \mathbb{R}. 
\end{align}
Considering $b \to 0$ in (\ref{uhyp}) one easily obtains a necessary condition on 4-tuple $(m,p,q,r)$, namely, 
\begin{align}\label{necessary}
|r| \leq \sqrt{\frac{p+m}{q+m}}. 
\end{align}
If $m=-1$  then we are in the setting of  a celebrated theorem of Bonami~\cite{bonami70}, also known as Bonami--Beckner--Gross ``two-point inequality'', which says that (\ref{necessary}) implies  (\ref{uhyp}) when  $(m,p,q,r)=(-1,p,q,r)$ and $q\geq p > 1$. A theorem of Weissler \cite{weissler80} shows that (\ref{necessary}) implies (\ref{uhyp}) when  $(m,p,q,r)=(0,p,q,r)$ and  $q\geq p>0$.  Extension of (\ref{inequality_w}) to higher dimensions, the main theorem of our paper,  in an equivalent way can be restated as  (\ref{necessary})  implies (\ref{uhyp}) when $(m,p,q,r)=(n-2, p, 2, r)$ with $n \geq 2$, $n \in \mathbb{N}$, and $2\geq p>0$.  In Section~\ref{ultra}, using log-Sobolev inequalities for ultraspherical measures,  we show that  (\ref{necessary}) implies (\ref{uhyp}) for the 4-tuples $(m,p,q,r)$ with  $q\geq p \geq 6$ and all $m\geq -1$. Despite of partial progresses the description of all 4-tuples $(m,p,q,r)$ for which the hypercontractivity (\ref{uhyp}) holds true remains an open question.

\section{Main Theorem}
$\Sup^{n-1}$ denotes the unit sphere in $\RR^n$, $\sigmaup$ the normalized Haar measure on $\Sup^{n-1}$, $\Bup^n_r(x)$ denotes the open ball in $\RR^n$ with radius $r>0$, centered at $x\in \RR^n$,  
and we set for convenience, $\Bup^n_r = \Bup^n_r(0).$
\begin{theorem} \label{th-main} Let $n \in \NN $ with $ n \ge 2 . $
 Let $p\in (0,2]$ and $\lambda \le \frac {n+p-2}{n}$. Then,
  \begin{align} \label{inequality_thm}
    \int_{\Sup^{n-1}} |x-a z|^p \D \sigmaup(z) \ge \left(|x|^2 + \lambda a^2 \right)^{\frac p 2} \qquad \text{for } x \in \RR^n, a\in [0,\infty),
 \end{align}
 and $\frac{n+p-2}{n}$ is the best (i. e. largest) constant satisfying \eqref{inequality_thm}.
\end{theorem}

We move to the front the elementary observation  that  $\frac{n+p-2}{n}$ is the best (i. e. largest) constant satisfying \eqref{inequality_thm}.
For  $x \in \RR^2 $ with  $|x|=1$,
and   $ a,  \lambda \in \RR ^+$, define
 \begin{equation}
  I(a)= \int_{\Sup^{n-1}} |x- a z|^p \D \sigmaup(z), \qquad \text{and}\qquad 
  g(a)= \left(1+\lambda a^2 \right)^{\frac p 2},
 \end{equation}
  Assuming that   \eqref{inequality_thm} holds true, for $\lambda > 0 $ we have 
 \begin{align}\label{assumption1}
  I(a) \ge g(a) \qquad \text{for } a \ge 0.
 \end{align}
 We now show that \eqref{assumption1} implies that  $\lambda \le \frac {n+p-2}{n}$. 
Clearly we have $ I(0)=1 $,   $g(0)=1,$    $g'(0)=0$ and $ g ''(0) = p/\lambda . $ 
Next, since   
 \begin{align*}
  \partialup_a |x-a z|^p = p |x-a z|^{p-2} z \cdot (a z-x) 
 \end{align*}   we have    $I'(0)=0. $ 
Hence  \eqref{assumption1} implies that $ I''(0) \ge g''(0)$. 
Calculating further
  \begin{align*}
  \partialup_a^2 |x-az|^p = p(p-2) |x-az|^{p-4} (z\cdot (az-x))^2 + p|x-az|^{p-2} |z|^2,
 \end{align*}
and invoking the integral identity  
$$\int_{\Sup^{n-1}} |( x \cdot z)|^2 d\sigma(z) = \frac{1}{n} $$
 gives 
$  I''(0)=\frac{p(p-2)}{n}+p. $
 Thus  $ I''(0) \ge g''(0)$, implies that  $\lambda \le \frac{n+p-2}{n}$.

 Before turning to the proof of Theorem~\ref{th-main} we
determine  the parameters $ n $ and $q$ for which
  $x\mapsto |x|^q$ is a subharmonic mapping on $\RR^n $, and draw consequences  (analogous to Jensen's formula in complex analysis).

\begin{lemma} \label{subharmonic_estimate}

 Let $n\in \NN$ and $q\in \RR$. The function $f \colon \RR^n\setminus\{0\}\to \RR, x\mapsto |x|^q$ is subharmonic, if and only if $q\ge \max\{0,2-n\}$ or $q\le \min\{0,2-n\}$,
 and then 
 \begin{equation}\label{7-3-20-1}
  \int_{\Sup^{n-1}} |x-a z|^{q} \D \sigmaup(z) \ge \max\{a,|x|\}^q \qquad \text{for } a\in \RR, x\in \RR^n.
 \end{equation}

\end{lemma}

\begin{proof}
 For $i\in\{1,\ldots,n\}$ we have
 $$
  \partialup_i f(x) = q x_i |x|^{q-1}, \qquad \partialup_i^2 f(x) = q|x|^{q-2} + q(q-2) x_i^2 |x|^{q-4}
 $$
 and therefore
  \begin{equation}\label{7-3-20-4}
  \Deltaup f(x) = q(n+q-2)|x|^{q-2}.
 \end{equation}
Clearly the sign of the factor $q(n+q-2)$ determines if $f$ is subharmonic or not.

We next turn to verifying that $q(n+q-2)\ge 0 $ implies \eqref{7-3-20-1}. 
 If $a<|x|$, the mean value property of subharmonic functions directly yields
 $$
  \int_{\Sup^{n-1}} |x-a z|^q \D \sigmaup(z) \ge |x|^q.
 $$
 To treat the case $a>|x|$, we define $H_a\colon \Bup_a^n(0) \to \RR$ by
 $$
  H_a(x) := \int_{\Sup^{n-1}} |x-a z|^q \D \sigmaup(z).
 $$
 and notice, that $H_a$ is subharmonic and rotational invariant, i. e. there exists a function $h_a\colon [0,a) \to \RR$, such that
 $$
  H_a(x) = h_a(|x|) \qquad \text{for } x \in[0,a).
 $$
 Using subharmonicity and rotational invariance together with the representation of the Laplace operator in $n$-dimensional spherical coordinates, we obtain
 $$
  0 \le \Deltaup H_a(x) = |x|^{1-n}\partialup_r \left(r\mapsto r^{n-1} \partialup h_a(r)\right)(|x|) \qquad \text{for } |x| \in (0,a).
 $$
 This yields
 $$
  r^{n-1} \partialup h_a(r) \ge 0 \qquad \text{for } r \in [0,a)
 $$
 and consequently
 $$
  h_a(r)\ge h_a(0) = a^q \qquad \text{for } r \in [0,a).
 $$
 Hence for $a>|x|$ we have  $H_a(x) = h_a(|x|)\ge  a^q, $ and hence 
 $$
 \int_{\Sup^{n-1}} |x-a z|^q \D \sigmaup(z) \ge a^q.
 $$
\end{proof}

 \begin{proof}  
   We now prove that  \eqref{inequality_thm} holds true for $\lambda := \frac {n+p-2} n$.
Since the case  $n=2$ is already known,  we consider  $n \ge 3$.
   An application of the divergence theorem yields that 
 \begin{equation}\label{7-3-20-2}
    \int_{\Sup^{n-1}} |x- a z|^p \D \sigmaup(z) = 1+ a^2 p (p+n-2) \int_0^1\int_0^t \left(\frac r t \right)^{n-1} \int_{\Sup^{n-1}} |x- a z|^{p-2} \D \sigmaup(z) \D r \D t.
 \end{equation}
Indeed, put $f\colon \RR^n\to \RR, f(y) := |x-a y|^p$ and define a vectorfield $X$ by $X(y) := \nabla f(t y)$. Then $\mathrm{div} X(y) = t \Delta f(t y)$ and by the divergence theorem
\begin{equation} \label{7-3-20-3}
    \begin{aligned}
\partialup_t \int_{\Sup^{n-1}} f(t z) \D \sigmaup(z) &= \int_{\Sup^{n-1}} X(z)\cdot z \D \sigmaup(z) \\
     &=\frac 1 {n \mathrm{Vol}_n\left(\Bup_2^n\right)} t \int_{\Bup_2^n} \Delta f(t y)\D y \\
     &= \frac 1 {n \mathrm{Vol}_n\left(\Bup_2^n\right)} \int_{t \Bup_2^n} \Delta f(y) \D y \\
     &= \frac 1 {t^{n-1}} \int_0^t\int_{\Sup^{n-1}} r^{n-1} \Delta f(r z)\D r \D \sigmaup(z).
    \end{aligned}
\end{equation}
  Integrating the identity \eqref{7-3-20-3} from $ t=0$ to $t=1$ and invoking  \eqref{7-3-20-4}
gives \eqref{7-3-20-2}.  
 Define
 $$
  H(a,x):= \int_{\Sup^{n-1}} |x-a z|^{p-2} \D \sigmaup(z).
 $$
 Then $H(a,\cdot)$ is rotational invariant, i. e. there exists a function $h\colon [0,\infty)^2 \to \RR$, such that $H(a,x) = h\left(a,|x|\right)$. By \eqref{7-3-20-2} and re-scaling  
we have 
 \begin{align} \label{identity2}
    \int_{\Sup^{n-1}} |x- a z|^p \D \sigmaup(z) = 1+ p (p+n-2)\int_0^a\int_0^t t^{1-n} u^{n-1} h(u,1) \D u \D t.
 \end{align}
The proof of Theorem\ref{th-main} will be obtained by proving suitable lower estimates for the 
volume integral appearing  on the right hand side of \eqref{identity2}.    
We will  distinguish the case where  $x\mapsto |x|^{p-2}$ is sub-harmonic 
(corresponding to $n=3$ and   $p\le 1$), and the case where sub-harmonicity fails 
(corresponding to   $n\ge 4$ or $p>1$).
%

 \subsection{Case $n=3$ and  $p\le 1$}

 First note that

\begin{equation}\label{28-2-20-1}
  1 +p(p+1) \int_0^a\int_0^t t^{-2} u^{2} \max\{ 1,u \}^{p-2} \D u \D  t= \begin{cases}
    1+\frac{p(p+1)}{6} a^2,  & a \in [0,1] \\
    a^p + \frac{p(2-p)}{3a} +\frac{(p-1)p}{2}, & a > 1
  \end{cases}.
\end{equation}
 Indeed, \eqref{28-2-20-1} follows from a direct calculation separating the cases 
$ a\le 1$ and $a > 1. $  
\begin{description}
\item[Case $a\le 1$] We calculate
  \begin{equation*}
    \int_0^a \int_0^t t^{-2} u^{2} \max\{1,u\}^{p-2} \D u\D t = \int_0^a t^{-2} \int_0^t u^{2} \D u \D t = \frac{a^2}{6},
  \end{equation*}
  which yields \eqref{28-2-20-1} for $a\le 1$.
\item[Case $a>1$] We calculate
  \begin{align*}
    &\int_0^a \int_0^t t^{-2} u^{2} \max\{1,u\}^{p-2} \D u\D t \\
    &= \int_0^1t^{-2} \int_0^t u^{2} \D u \D t + \int_1^a t^{-2} \int_0^1 u^{2} \D u \D t + \int_1^a t^{-2} \int_1^t u^{p} \D u \D t \\
    &= \frac 1 {6} - \frac{a^{-1} - 1}{3} + \frac {a^p-1}{(p+1)p}+ \frac {a^{-1}-1} {(1+p)},
  \end{align*}
  which yields  \eqref{28-2-20-1} for $a\ge 1$, by arithmetic.
\end{description}
Since  $x\mapsto |x|^{p-2}$ is subharmonic, for  $n=3$ and $p\in(0,1]$,  
Lemma \ref{subharmonic_estimate} yields
$h(a,x) \ge \max\{1,a\}^{p-2}$. Applying this estimate to
\eqref{identity2} and invoking \eqref{28-2-20-1}
we obtain
 \begin{equation}\label{10-3-2020-3}
   \int_{\Sup^{2}} |x-a z|^p \D \sigmaup(z) \ge \begin{cases}
     1+\frac{p(p+1)} 6 a^2,  & a \le 1 \\
     a^p + \frac{p(2-p)}{3 a} - \frac{p(1-p)}{2}, & a > 1
   \end{cases}.
  \end{equation}
   Defining
   $$
   g(a) := \begin{cases}
     1+\frac{p(p+1)} 6 a^2,  & a \in [0,1] \\
     a^p + \frac{p(2-p)}{3 a} - \frac{p(1-p)} 2 , & a^2 \in \left( 1, \frac 3 {2-p} \right) \\
     a^p, & a^2 \ge \frac 3 {2-p}
   \end{cases},
   $$
   it suffices to show
  \begin{equation}\label{10-3-2020-1}
   \int_{\Sup^{2}} |x-a z|^p \D \sigmaup(z) \ge g(a) \ge \left(1+
     \frac {p+1} 3 a^2 \right)^{\frac p 2}.
   \end{equation}
We first consider  $a^2\ge \frac 3 {2-p}$. 
  In that case we have 
 \begin{equation}\label{10-3-2020-2}
 \int_{\Sup^{2}} |x- a z|^p \D \sigmaup(z) \ge a^p \ge \left(
   1+\frac{p+1}{3} a^2 \right)^{\frac p 2} .
   \end{equation}
 Indeed, by Lemma \ref{subharmonic_estimate}, $x\mapsto |x|^p$ is
 subharmonic. Taking into account that $|x | = 1 $ and
 $ a^2 > \frac 3 {2-p} $, Lemma \ref{subharmonic_estimate} yields
 $$
 \int_{\Sup^{n-1}} |x- a z|^p \D \sigmaup(z) \ge \max\{a,1\}^p \ge a ^
 p .
 $$
To obtain the  second estimate in \eqref{10-3-2020-2}  note that 
 $
 a^2 \ge \frac 3 {2-p}$ holds {if and only if} $ a^2 \ge 1+
 \frac{p+1}{n} a^2.
 $

We now turn to the case  $a^2 <  \frac 3 {2-p}$.
 By \eqref{10-3-2020-3} in this case it remains to show the second inequality of \eqref{10-3-2020-2}.
If moreover  $a\in [0,1]$ this is just Bernoulli's
 inequality. If  finally 
 $a^2 \in \left(1,\frac 3 {2-p} \right)$ we proceed as follows: For
 $p\in (0,1]$, we define
 $$
 \phi(t) := t^{\frac p 2} + \frac{p(2-p)}{3\sqrt t} - \frac{p(1-p)} 2
 - \left(1+\frac {p+1} 3 t \right)^{\frac p 2}.
 $$
 We show that $\phi(t) \ge 0$ for
 $t\in \left( 1, \frac 3 {2-p} \right)$. Indeed, since
 $t<\frac 3{2-p}$ holds if and only if $t < 1+\frac{p+1} 3 t$, we get
 \begin{align*}
   \phi'(t) &= \frac p 2 t^{\frac{p-2} 2} - \frac {p(2-p)} 6 t^{-\frac 3 2} - \frac{p(p+1)}{6}\left(1+\frac{p+1} 3 t \right)^{\frac {p-2} 2} \\
            &\ge \left( \frac p 2 - \frac{p(p+1)} 6 \right) t^{\frac {p-2} 2} - \frac{p(2-p)} 6 t^{-\frac 3 2} = \frac {p(2-p)} 6 \left(t^{\frac{p-2} 2}-t^{-\frac 3 2}\right) \ge 0.
 \end{align*}
 Due to $\phi(1)\ge 0$, this implies $\phi(t) \ge 0$ for
 $t\in\left(1,\frac 3 {2-p} \right)$.  Summing up for $p\in(0,1]$ and
 $t= a^2 \in \left( 1,\frac 3{2-p} \right)$ we have
 $$
 a^p + \frac{p(2-p)}{3 a} - \frac{p(1-p)} 2 \ge \left( 1+\frac{p+1} 3
   a^2  \right)^{\frac p 2}.
 $$
 
 \begin{figure}[htp]
    \centering
    \includegraphics[width=.5\textwidth]{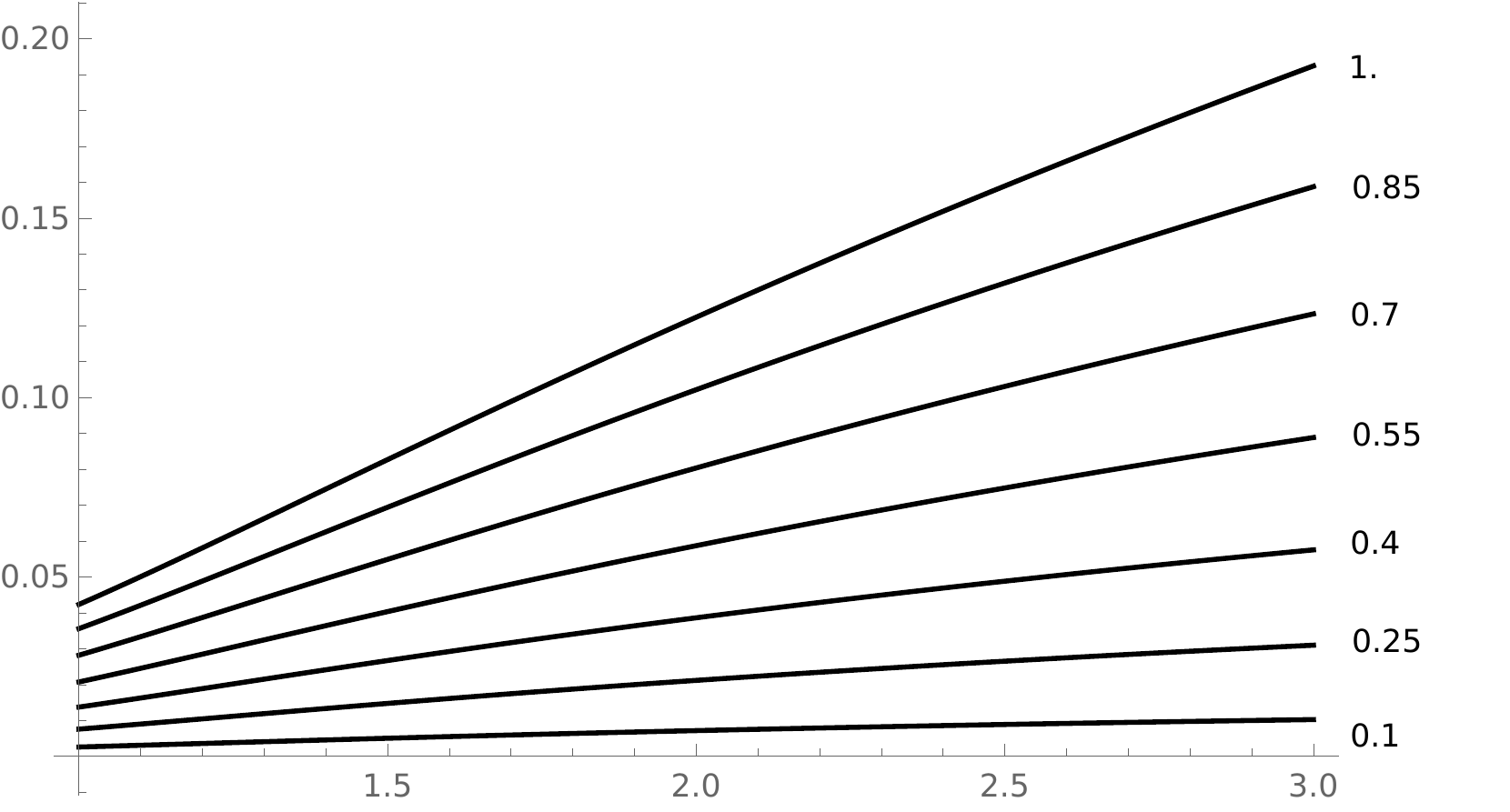}\hfill 
    \includegraphics[width=.5\textwidth]{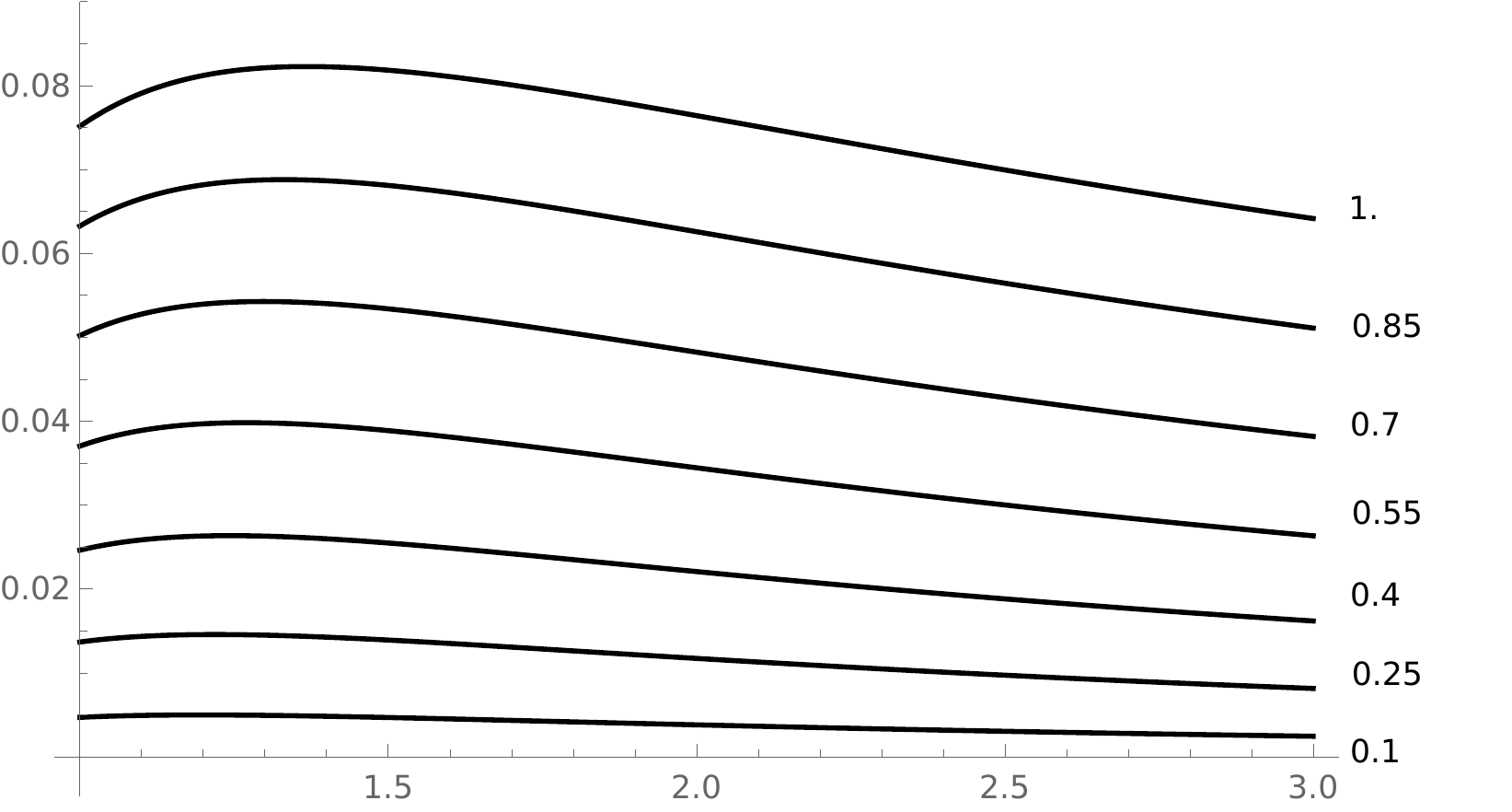}
    \caption{Plots of the functions $\phi$ and  $\phi'$ for $p\in\{0.1, \ldots,0.85, 1 \}$.}
    \label{figure1}
\end{figure}

 \subsection{Case $n> 3$ or $p>1$} Since we cannot apply Lemma
  \ref{subharmonic_estimate}, we need another lower bound for
  $h(a,1)$. In order to accomplish that we use the formula
  \begin{align} \label{formula_1} r^{-\zeta} = \frac 1
    {\Gammaup\left(\frac \zeta 2\right)} \int_0^\infty t^{-\frac \zeta
      2 -1} \exp\left( -\frac {r^2} t \right) \D t = \frac 1
    {\Gammaup\left(\frac \zeta 2\right)} \int_0^\infty t^{\frac \zeta
      2 -1} \exp\left( -r^2 t \right) \D t,
  \end{align}
  which holds for all $r>0$ and $\Re \zeta>0$. Putting $\zeta := 2-p$,
  i. e. $p=2-\zeta$, we get
  \begin{align*}
    H(a,x) &=\int_{\Sup^{n-1}} |x-a z|^{p-2} \D \sigmaup(z) \\
           &= \frac 1 {\Gammaup\left(\frac \zeta 2\right)} \int_0^\infty \int_{\Sup^{n-1}} t^{-\frac \zeta 2 -1}\exp\left(-\frac{|x-a z|^2} t \right) \D \sigmaup(z)\D t \\
           &= \frac 1 {\Gammaup\left(\frac \zeta 2\right)} \int_0^\infty \left( \int_{\Sup^{n-1}} \exp\left(\frac{2ax\cdot z}{t}\right) \D \sigmaup(z)\right) t^{-\frac \zeta 2 -1} \exp\left(-\frac{1+a^2} t\right)\D t.
  \end{align*}
  We are thus left with finding a good lower bound for
   $$
   \int_{\Sup^{n-1}} \exp(\lambda x\cdot z)\D \sigmaup(z) =
   \int_{\Sup^{n-1}} \cosh\left(\lambda |x\cdot z| \right) \D
   \sigmaup(z),
   $$
   where $\lambda>0$ and $|x|=1$.  The obvious bound is $1$, which
   eventually turns out not to be sufficient for $p < \frac 4{n+2}$,
   so we take the second Taylor approximation:
   $\cosh s \ge 1+\frac {s^2} 2$. By \eqref{formula_1} and the
   functional equation of the gamma function we conclude
   \begin{align*}
     h(a,1) &\ge \left(1+a^2\right)^{-\frac \zeta 2} + \frac{2 a^2}{n\Gammaup\left(\frac \zeta 2\right)} \int_0^\infty t^{-\frac \zeta 2 -3}\exp\left(-\frac{1+a^2} t \right) \D t \\
            &= \left(1+a^2\right)^{-\frac \zeta 2} + \frac{2a^2 \Gammaup\left(\frac \zeta 2 + 2\right)}{n\Gammaup\left(\frac \zeta 2  \right)}\left(1+a^2 \right)^{-\frac \zeta 2 -2} \\
            &= \left(1+a^2\right)^{\frac p 2 -1} \left(1+\frac{(4-p)(2-p)}{2n} \frac{a^2}{\left(1+a^2 \right)^2} \right) =: \psi(a).
   \end{align*}
   According to \eqref{identity2} it remains to prove that
   $$
   1+ p(p+n-2) \int_0^a\int_0^t t^{1-n} u^{n-1} \psi(u) \D u \D t \ge
   \left(1+\frac{p+n-2} n a^2\right)^{\frac p 2}.
   $$
   We set $c := \frac{n+p-2} n$ and show
   $$
   F(a) := 1+ p(p+n-2) \int_0^a\int_0^t t^{1-n} u^{n-1} \psi(u) \D u
   \D t - \left(1+c a^2\right)^{\frac p 2} \ge 0,
   $$
   Since $F(0)=0$, this follows from $F' \ge 0$, i. e.
   \begin{equation*}
     n\int_0^a u^{n-1}\psi(u) \D u - a^n(1+c a^2)^{\frac p 2-1}\ge0,
   \end{equation*}
   which in turn follows from
   $$
   n a^{n-1} \psi(a) -\partialup_a \left(a^n\left(1+c a^2
     \right)^{\frac p 2 -1}\right)\ge 0.
   $$
   Rearranging terms this amounts to
$$
1+\frac{(4-p)(2-p)a^2}{2n\left(1+a^2\right)^2} -\left(\frac{1+a^2}{1+c
    a^2}\right)^{1-\frac p 2} +\frac{a^2c(2-p)}{n\left(1+a^2\right)}
\left(\frac{1+a^2}{1+c a^2}\right)^{2-\frac p 2} \ge0.
$$
Put $x:=(1+c a^2)/(1+a^2)$, then $x\in(c,1)$ and
$$
a^2=\frac{1-x}{x-c},\quad 1+a^2=\frac{1-c}{x-c},\quad\mbox{and}\quad
\frac{a^2}{1+a^2}=\frac{1-x}{1-c}.
$$
Thus we have to show that
$$
1+\frac{(4-p)(2-p)(1-x)(x-c)}{2n(1-c)^2} -x^{\frac p 2-1}
+\frac{c(2-p)(1-x)}{n(1-c)} x^{\frac p 2 -2} \ge 0
$$
i.e.
$$
x^{2-\frac p 2} \geq\frac{1+\frac{n(1-c)+c(2-p)}{n(1-c)}(x-1)}
{1+\frac{(4-p)(2-p)(1-x)(x-c)}{2n(1-c)^2}}
=\frac{1-c-(1-c^2)(1-x)}{1-c+\left(2-\frac p 2 \right)(1-x)(x-c)}.
$$
Considering $n\ge 4$ or $p > 1$, we have
$c=1-\frac{2-p}{n} \ge \frac 1 2$. So eventually it suffices to prove
that given $q:= 2-\frac p 2 \in [0,1]$, then for all
$(x,y)\in [0,1]^2$ satisfying $x\ge y \ge \frac 1 2$, we have
\begin{equation}\label{ee6}
x^q(1-y+q(1-x)(x-y))\ge 1-y-(1-y^2)(1-x).
\end{equation}
The function $q\mapsto x^q(1-y+q(1-x)(x-y))$ is decreasing. Indeed,
the derivative of the logarithm with respect to $q$ is
$$
\frac{(1-x)(x-y)}{1-y+q(1-x)(x-y)} - \log \frac 1 x\le
\frac{(1-x)(x-y)}{1-y} - \log \frac 1 x \le 1-x-\log\frac 1 x \le 0,
$$
where we simply used the fact $y\le x\le 1$. Thus we only have to
prove, that, assuming $\frac 1 2\le y \le x \le 1$, he have
$$
x^2(1-y+2(1-x)(x-y)) -1+y+(1-y^2)(1-x)\ge 0.
$$
The left-hand side is a polynomial in $x$ of order $4$, which
factorizes to
$$
(1-x)(x-y)(2x^2+y-1).
$$
Due to the conditions on $x$ and $y$, this is obviously non-negative.

\end{proof}
However the polynomial is negative for $y\leq x<1/2$ and thus
the inequality (\ref{ee6}) doesn\rq t\ hold for small values of $p$
and $n\in\{2,3\}$. Hence the above argument does not apply to dimension
two and three!

 \begin{figure}[htp]
    \centering
    \includegraphics[width=.5\textwidth]{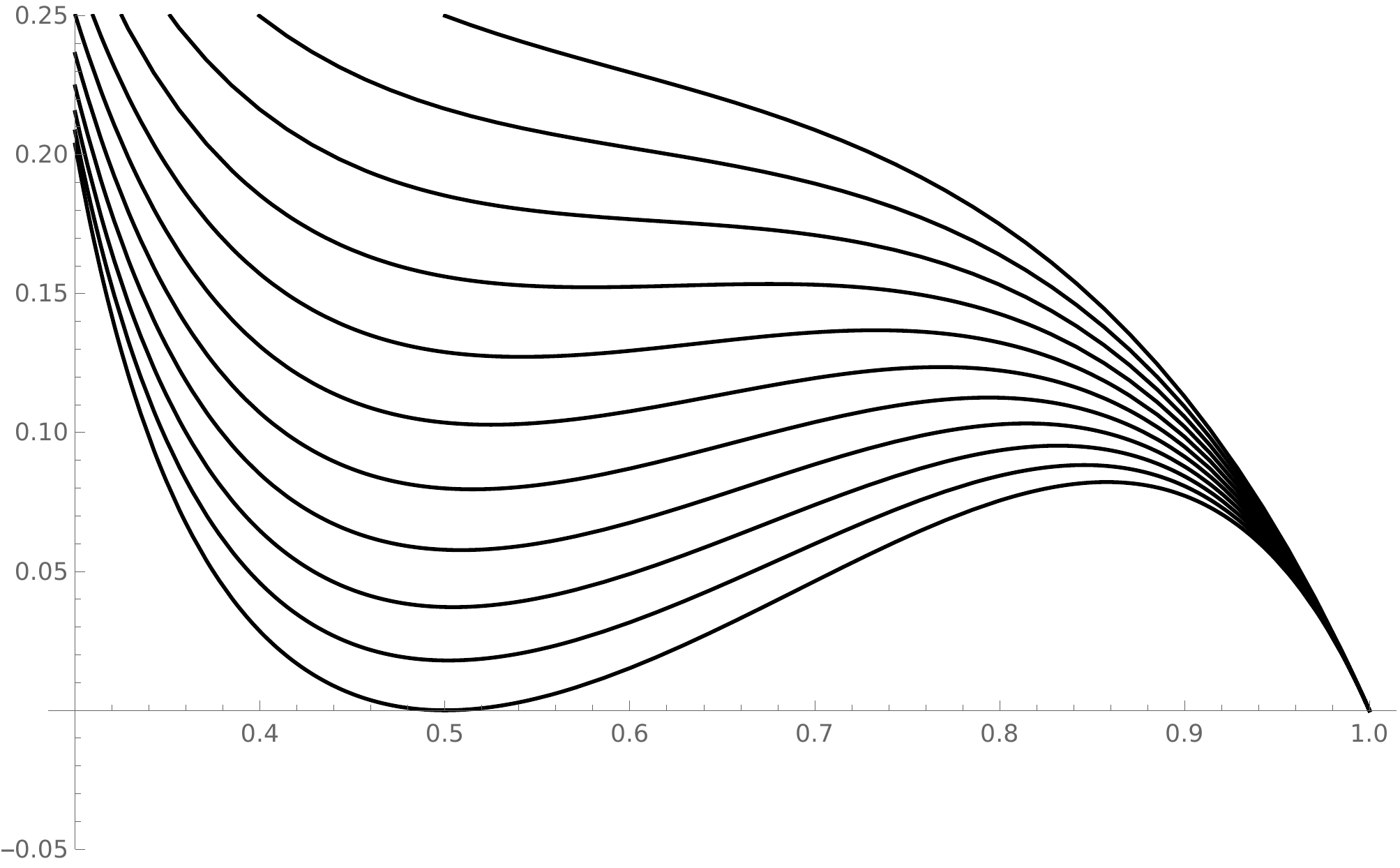}\hfill 
    \includegraphics[width=.5\textwidth]{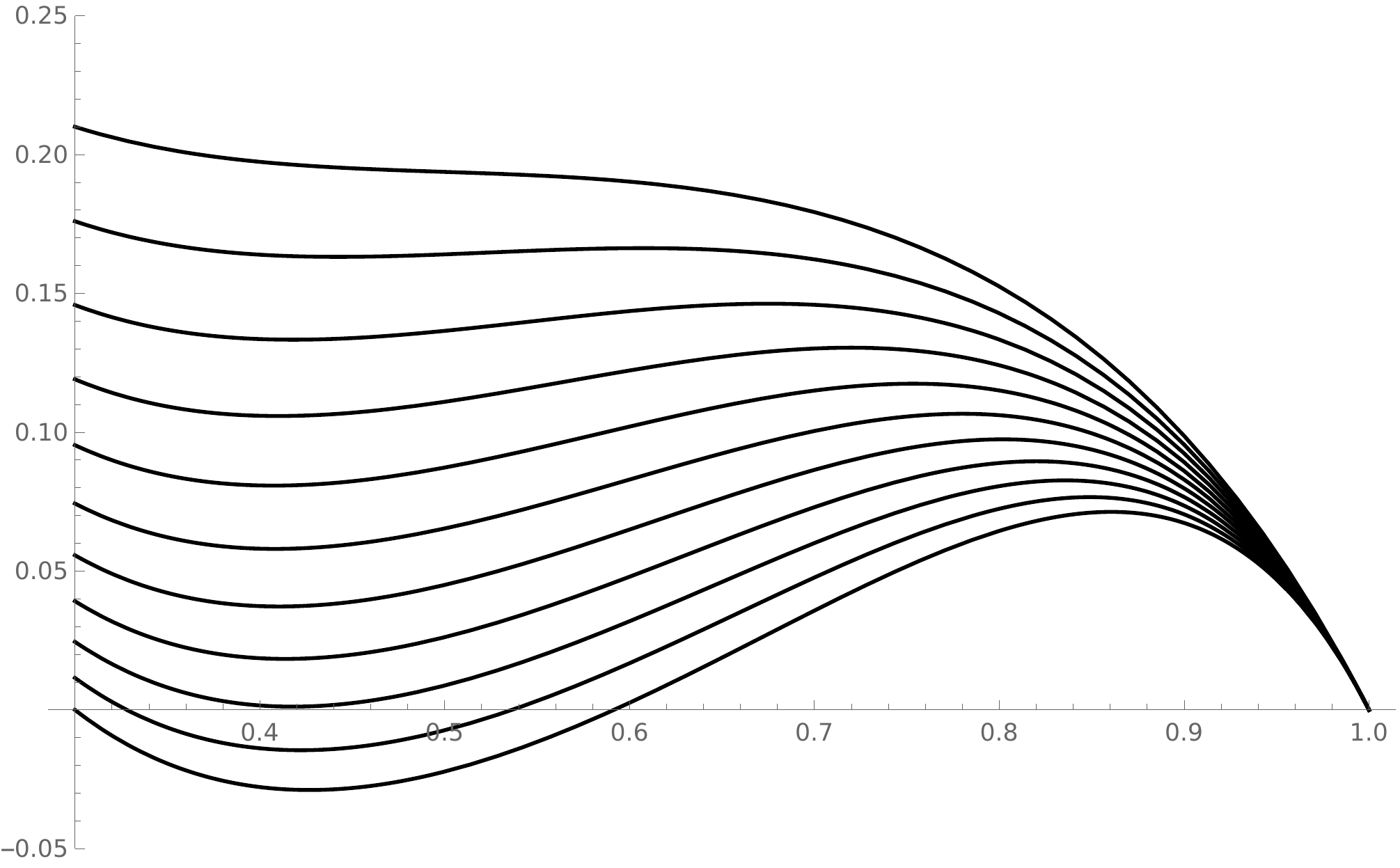}
    \caption{Plots of $x\mapsto x^q-\frac{1-y-(1-y^2)(1-x)}{1-y+q(1-x)(x-y)}$ for $y\in\{0.5,0.3\}$ and $q\in\{1,1.1,\ldots,2\}.$} 
    \label{figure2}
\end{figure} 


\section{Hypercontractivity for ultraspherical measures on the unit circle}\label{ultra}

The main theorem  (Theorem~\ref{th-main}) in an equivalent form we can restate as   
 for  $0<p\le 2 $, $n\geq 2$, $|r|\leq \sqrt{\frac{p+n-2}{2}}$, $r \in \mathbb{R}$,  we have 
  \begin{equation}\label{rewrite}
   \| x+ray\|_{L^{2}(\mathbb{S}^{n-1}, \, d\sigma(y))} \leq \| x+ay\|_{L^{p}(\mathbb{S}^{n-1}, \, d\sigma(y))} \quad \text{for all} \quad x \in \mathbb{R}^{n}, \, a \in \mathbb{R}. 
     \end{equation}

    In this section we consider an extension of (\ref{rewrite}), namely, let $n \geq 2$, and $0<p\leq q <\infty$. We are interested to find the largest possible constant $C=C(n,p,q)>0$ such that for all $r \in \mathbb{R}$, $|r|\leq C(p,q,r)$ we have 
   \begin{align}\label{ax1}
   \| x+ary\|_{L^{q}(\mathbb{S}^{n-1}, \, d\sigma(y))} \leq \| x+ay\|_{L^{p}(\mathbb{S}^{n-1}, \, d\sigma(y))} \quad \text{for all} \quad x \in \mathbb{R}^{n}, \, a \in \mathbb{R}. 
   \end{align}

   First we prove a theorem on the unit circle for ultraspherical measures
   \begin{align*}
   d\nu_{m}(z) =  c_{m} |\sin(\theta)|^{m}d\theta \quad \text{for all real} \quad  m> -1, 
   \end{align*}
   where $z = e^{i \theta} \in \mathbb{S}^{1}$, and the scalar $c_{m} \stackrel{\mathrm{def}}{=} \frac{\Gamma(m/2+1)}{2\Gamma(1/2) \Gamma(m/2 +1/2)}$ is chosen in such a way that $\nu_{m}(\mathbb{S}^{1})=1$. For $m=-1$ we set $d\nu_{-1}(z) = \frac{1}{2}(\delta_{-1}(z)+\delta_{1}(z))$. 
   \begin{theorem}\label{mtav2}
Let  $m \geq -1$ and $6 \leq p \leq q$.  We have 
\begin{align}\label{ax2}
\| 1+rbz\|_{L^{q}(\mathbb{S}^{1}, d\nu_{m})} \leq \| 1+bz\|_{L^{p}(\mathbb{S}^{1}, d\nu_{m})} \quad \text{for all} \quad b \in \mathbb{R}. 
\end{align}
if and only if $|r| \leq \sqrt{\frac{p+m}{q+m}}$. 
   \end{theorem}
   Let us show that the theorem implies 
   \begin{corollary}
For any $6 \leq p \leq q$, all integers $n \geq 2$, and any real $|r|\leq \sqrt{\frac{p+n-2}{q+n-2}}$ the inequality (\ref{ax1}) holds true. 
   \end{corollary}
   Indeed, without loss of generality we can assume $|x|=1$ in (\ref{ax1}). Next, for  $y=(y_{1}, \ldots, y_{n}) \in \mathbb{S}^{n-1}$ and $\lambda = \frac{n-2}{2}$,   we have 
  \begin{align*}
  &\| x+ay\|^{p}_{L^{p}(\mathbb{S}^{n-1}, \, d\sigma(y))} = \int_{\mathbb{S}^{n-1}} (1+2a \langle x, y\rangle + a^{2})^{p/2} d\sigma(y)=\\
  &\frac{\Gamma(\lambda+1)}{\Gamma(1/2) \Gamma(\lambda +1/2)}\int_{-1}^{1} (1 +2at+a^{2})^{p/2}(1-t^{2})^{\lambda - (1/2)} dt \stackrel{(t=\cos(\theta))}{=} \\
  &\frac{\Gamma(\lambda+1)}{\Gamma(1/2) \Gamma(\lambda +1/2)}\int_{0}^{\pi} (1 +2a\cos(\theta)+a^{2})^{p/2}\sin^{2\lambda}(\theta) d\theta=\\
  &\int_{\mathbb{S}^{1}}| 1+az|^{p} d\nu_{2\lambda}(z)=\| 1+az\|^{p}_{L^{p}(\mathbb{S}^{1}, d\nu_{n-2})}
  \end{align*}
 Similarly we have $\| x+ray\|_{L^{q}(\mathbb{S}^{n-1}, \, d\sigma(y))} = \| 1+raz\|_{L^{q}(\mathbb{S}^{1}, d\nu_{n-2})}$. Thus the inequalities (\ref{ax1}) and (\ref{ax2}) are the same with $m=n-2$.

\vskip0.5cm

Next we prove Theorem~\ref{mtav2}.

\begin{proof}
As the measure $d\nu_{-1}(z) = \frac{1}{2}\left( \delta_{-1}(z)+\delta_{1}(z)\right)$ is the weak* limit of of the measures $d\nu_{m}(z)$ when $m \to -1$, $m>-1$, without loss of generality we can assume that $m>-1$ in the theorem.  

First we show that the assumption $|r| \leq \sqrt{\frac{p+m}{q+m}}$ is necessary for the hypercontractivity (\ref{ax2}). Indeed,  notice that 
\begin{align*}
\int_{\mathbb{S}^{1}} (\Re(z))^{2} d\nu_{m}(z) = c_{m} \int_{0}^{2\pi} \cos^{2}(\theta) |\sin(\theta)|^{m} d\theta = 1- \frac{c_{m}}{c_{m+2}} = 1-\frac{m+1}{m+2}=\frac{1}{m+2}.
\end{align*}
Therefore
\begin{align*}
 &\| 1+bz\|_{L^{p}(\mathbb{S}^{1}, d\nu_{m})} = \left( \int_{\mathbb{S}^{1}} |1+bz|^{p} d\nu_{m}(z)\right)^{1/p} = \left( \int_{\mathbb{S}^{1}} \left(1+2b \Re(z)+b^{2} \right)^{p/2} d\nu_{m}(z)\right)^{1/p} = \\
 & \left( \int_{\mathbb{S}^{1}} 1+\frac{p}{2}(2b \Re(z)+b^{2}) + \frac{p}{4}(\frac{p}{2}-1) 4 b^{2} (\Re(z))^{2} +o(b^{2})d\nu_{m}(z)\right)^{1/p} =\\
 &\left(1+\frac{p}{2}b^{2} + \frac{p(p-2)}{2}b^{2} \int (\Re(z))^{2} d\nu_{m}  \right)^{1/p} = 1+\frac{b^{2}}{2} + \frac{p-2}{2}b^{2} \frac{1}{m+2}+o(b^{2}) \\
 &= 1+\frac{b^{2}}{2}\cdot \frac{m+p}{m+2} + o(b^{2}).
\end{align*}
So the inequality $\| 1+rbz\|_{L^{q}(\mathbb{S}^{1}, d\mu_{m})} \leq \| 1+bz\|_{L^{p}(\mathbb{S}^{1}, d\mu_{m})}$ implies $r^{2}\frac{m+q}{m+2} \leq \frac{m+p}{m+2}$. Since $p,q > -m$ we obtain  i.e., $|r| \leq \sqrt{\frac{m+p}{m+q}}$. 

Next we show that the necessary condition $|r| \leq \sqrt{\frac{p+m}{q+m}}$ is sufficient for (\ref{ax2}). Since $q\geq 1$ and $d\nu_{m}(z)$ is even measure, i.e., $d\nu_{m}(z)=d\nu_{m}(-z)$,  we see that the the map $r \mapsto \| 1+rbz\|_{L^{q}(\mathbb{S}^{1}, d\nu_{m})}$ is even convex function on $\mathbb{R}$, hence it is nondecreasing on $[0, \infty)$. Thus it suffices to prove (\ref{ax2}) in the case when $r = \sqrt{\frac{p+m}{q+m}}$. Let $m=2\lambda$. After rescaling $b$ as  $b \mapsto \frac{b}{\sqrt{p+m}}$,  we can rewrite (\ref{ax2}) as follows
\begin{align}\label{ax36}
 \left( \int_{-1}^{-1}  \left(1+\frac{2bt}{\sqrt{q+2\lambda}} + \frac{b^{2}}{q+2\lambda} \right)^{\frac{q}{2}} d\mu_{\lambda}(t) \right)^{\frac{1}{q}} \leq  \left( \int_{-1}^{-1}  \left(1+\frac{2bt}{\sqrt{p+2\lambda}} + \frac{b^{2}}{p+2\lambda} \right)^{\frac{p}{2}} d\mu_{\lambda}(t) \right)^{\frac{1}{p}}, 
\end{align}
where $d\mu_{\lambda}(t) = 2c_{2\lambda} (1-t^{2})^{\lambda - (1/2)}dt$ is a probability measure on $[-1,1]$.  Rescaling $b$ as $b \mapsto \frac{b}{\sqrt{2}}$ we see that the inequality (\ref{ax36}) simply means that the map 
  \begin{align*}
    s \mapsto  \left(\int_{-1}^{1}\left(1+\frac{2bt}{\sqrt{s+\lambda}} + \frac{b^{2}}{s+\lambda}\right)^{s} d\mu_{\lambda}(t) \right)^{1/s}
  \end{align*}
  is nonincreasing on $(3, \infty)$ (here $s=p/2$). If we differentiate in $s$,
  then after a certain calculation we see that it suffices to show the
  following log-Sobolev inequality: put
  $f(t) = 1+\frac{2bt}{\sqrt{s+\lambda}} + \frac{b^{2}}{s+\lambda}$, then 
  \begin{align*}
    &\int f^{s} \ln f^{s}  d\mu_{\lambda}- \int f^{s}  d\mu_{\lambda} \ln \int f^{s} d\mu_{\lambda} \leq  -s^{2} \int f^{s-1}  \frac{d}{ds}f  d\mu_{\lambda}=\\
    &s^{2} \int f^{s-1}  \left(bt (s+\lambda)^{-3/2}+ b^{2} (s+\lambda)^{-2} \right) d\mu_{\lambda}.
  \end{align*}
  Therefore, if we let $b (s+\lambda)^{-1/2} = \tilde{b}$ and 
  $g(t) = 1+2\tilde{b} t + \tilde{b}^{2}$, then our log-Sobolev
  inequality rewrites as follows
  \begin{align}\label{log}
    \int g^{s} \ln g^{s}  d\mu_{\lambda}- \int g^{s}  d\mu_{\lambda} \ln \int g^{s} d\mu_{\lambda} \leq \frac{s^{2}}{s+\lambda} \int g^{s-1}  \left(\tilde{b}t + \tilde{b}^{2}  \right) d\mu_{\lambda}.
  \end{align}
  The log-Sobolev inequality of Mueller--Weissler \cite[p 277]{mue-we82} for $d\mu_{\lambda}$
  states that

\begin{align}
  \int g^{s} \ln g^{s}  d\mu_{\lambda}- \int g^{s}  d\mu_{\lambda} \ln \int g^{s} d\mu_{\lambda}
  &\leq \frac{s^{2}}{2(2\lambda+1)} \frac{2\lambda+1}{2(\lambda+1)} \int (g')^{2}g^{s-2} d\mu_{\lambda+1} \nonumber\\
  &= \frac{s^{2}}{4(\lambda+1)} \int (g')^{2}g^{s-2} d\mu_{\lambda+1} \label{mw}.
\end{align}

Thus we need to show that
\begin{align*}
  \int (g')^{2}g^{s-2} d\mu_{\lambda+1} \leq \frac{4(\lambda+1)}{s+\lambda} \int g^{s-1}  \left(\tilde{b}t + \tilde{b}^{2}  \right) d\mu_{\lambda}.
\end{align*}
After an integration by parts we can
rewrite the left hand side of the last inequality as
$\frac{4(\lambda+1)}{s-1} \int g^{s-1} t\tilde{b} d\mu_{\lambda}$ (here we used the fact that $\frac{c_{2(\lambda+1)}}{c_{\lambda}} = \frac{\lambda+1}{\lambda+1/2}$).
Hence,   to prove  (\ref{ax2})  it suffices to show that  
\begin{align}\label{in02}
  \frac{1}{s-1} \int g^{s-1} t d\mu_{\lambda} \leq \frac{1}{s+\lambda} \int g^{s-1}  \left(t + \tilde{b}  \right) d\mu_{\lambda}
\end{align} 
holds true. We can rewrite (\ref{in02}) as 
\begin{align*}
\int g^{s-1} d\mu_{\lambda} \geq \int \frac{t(\lambda+1)}{b(s-1)} g^{s-1} d\mu_{\lambda}. 
\end{align*}
Integrating the right hand side by parts we see that it is enough to show 
\begin{align*}
\int (1+2at+a^{2})^{s-1} d\mu_{\lambda}(t) \geq \int (1+2at+a^{2})^{s-2} d\mu_{\lambda+1}(t),
\end{align*}
for all $a=\tilde{b}>0$. We claim that it suffices to consider the case when $a \in (0,1)$. Indeed, otherwise  we can write 

\begin{align*}
  \int (1+2at+a^{2})^{s-1} d\mu_{\lambda}(t) &= a^{2(s-1)}\int (a^{-2}+2a^{-1}t+1)^{s-1} d\mu_{\lambda}(t)\\
&\geq a^{2(s-1)}\int (a^{-2}+2a^{-1}t+1)^{s-2} d\mu_{\lambda+1}(t)  \\
&= a^{2} \int (1+2at+a^{2})^{s-2} d\mu_{\lambda+1}(t) \\ &\geq  \int (1+2at+a^{2})^{s-2} d\mu_{\lambda+1}(t).
\end{align*}
The inequality $\frac{s-1}{s-2}\geq 1$ implies 
\begin{align*}
\int (1+2at + a^{2})^{s-1} d\mu_{\lambda}(t) \geq \left( \int (1+2at+a^{2})^{s-2} d\mu_{\lambda}(t) \right)^{\frac{s-1}{s-2}}.
\end{align*}
Next, by Jensen's inequality we have $\int (1+2at+a^{2})^{s-2} d\mu_{\lambda}  \geq (1+a^{2})^{s-2}\geq 1$. Therefore 
\begin{align*}
\left( \int (1+2at+a^{2})^{s-2} d\mu_{\lambda}(t) \right)^{\frac{s-1}{s-2}} \geq \int (1+2at+a^{2})^{s-2} d\mu_{\lambda}(t). 
\end{align*}

Therefore, we need to show that   $\int (1+2at+a^{2})^{s-2} d\mu_{\lambda}(t) \geq \int (1+2at+a^{2})^{s-2} d\mu_{\lambda+1}(t)$. The inequality trivially holds true if $s=3$. 
Considering the linear function $F(t) = 1+2at+a^{2}$, it suffices to show that 
\begin{align}\label{fun01}
\int_{0}^{\infty}r^{s-3} \mu_{\lambda}(t \in [-1,1] : F(t) >r) dr \geq \int_{0}^{\infty}r^{s-3} \mu_{\lambda+1}(t \in [-1,1] : F(t) >r) dr.
\end{align}
Consider $h(u) = \mu_{\lambda}(t \in [-1,1] : t >u) - \mu_{\lambda+1}(t \in [-1,1] : t >u)$. Clearly $h(-1)=h(0)=h(1)=0$. Also 
\begin{align*}
h'(u) = -2c_{2\lambda} (1-u^{2})^{\lambda-1/2} + 2c_{2\lambda+2} (1-u^{2})^{\lambda+1/2} = 2c_{2\lambda+2}(1-u^{2})^{\lambda -1/2}\left( \frac{1}{2(\lambda+1)}-u^{2}\right).
\end{align*}
It follows that $h(u) \leq 0$ on $[-1,0]$ and $h(u) \geq 0$ on $[0,1]$. Therefore $\varphi(r) =  \mu_{\lambda}(t \in [-1,1] : F(t) >r) - \mu_{\lambda+1}(t \in [-1,1] : F(t) >r)$ changes sign only once i.e.,  there exists $r_{0} \in [0,\infty)$ such that $\varphi(r) \leq 0$ on $[0,r_{0}]$ and $\varphi(r) \geq 0$ on $[r_{0}, \infty)$. If $r_{0}=0$ then (\ref{fun01}) trivially holds true. If $r_{0}>0$, then we have
\begin{align}\label{in04}
\int_{0}^{\infty}\left( \left(\frac{r}{r_{0}}\right)^{s-3} - 1\right) \varphi(r) dr \geq 0 
\end{align}
because the integrand has nonnegative sign. Therefore, inequality (\ref{in04}) together with $\int_{0}^{\infty} \varphi(r)dr=0$ implies  $\int_{0}^{\infty} r^{s-3} \varphi(r)dr\geq 0$.

%
%
%

\end{proof}
\bibliographystyle{amsalpha} 
\bibliography{references}

\providecommand{\bysame}{\leavevmode\hbox to3em{\hrulefill}\thinspace}
\providecommand{\MR}{\relax\ifhmode\unskip\space\fi MR }
\providecommand{\MRhref}[2]{%
  \href{http://www.ams.org/mathscinet-getitem?mr=#1}{#2}
}
\providecommand{\href}[2]{#2}
\begin{thebibliography}{DGTJ84}

\bibitem[Ale07]{aleksandrov07}
A.~B. Aleksandrov, \emph{Spectral subspaces of the space {$L^p,\ p<1$}},
  Algebra i Analiz \textbf{19} (2007), no.~3, 1--75. \MR{2340705}

\bibitem[Bon70]{bonami70}
Aline Bonami, \emph{\'{E}tude des coefficients de {F}ourier des fonctions de
  {$L^{p}(G)$}}, Ann. Inst. Fourier (Grenoble) \textbf{20} (1970), no.~fasc.,
  fasc. 2, 335--402 (1971). \MR{283496}

\bibitem[DGTJ84]{davis1984}
William~J. Davis, D.~J.~H. Garling, and Nicole Tomczak-Jaegermann, \emph{The
  complex convexity of quasinormed linear spaces}, J. Funct. Anal. \textbf{55}
  (1984), no.~1, 110--150. \MR{733036}

\bibitem[LMS19]{alex2019sharp}
Alexander Lindenberger, Paul F.~X. Müller, and Michael Schmuckenschläger,
  \emph{Sharp complex convexity estimates}, arXiv:1901.07926 (2019).

\bibitem[MW82]{mue-we82}
Carl~E. Mueller and Fred~B. Weissler, \emph{Hypercontractivity for the heat
  semigroup for ultraspherical polynomials and on the {$n$}-sphere}, J.
  Functional Analysis \textbf{48} (1982), no.~2, 252--283. \MR{674060}

\bibitem[Wei80]{weissler80}
Fred~B. Weissler, \emph{Logarithmic {S}obolev inequalities and hypercontractive
  estimates on the circle}, J. Functional Analysis \textbf{37} (1980), no.~2,
  218--234. \MR{578933}

\end{thebibliography}


\end{document}